\documentclass[11pt]{article}
\pdfoutput=1 
\usepackage{a4}
\usepackage{amsmath}
\usepackage{amssymb}
\usepackage{amsthm}
\usepackage{lipsum}
\usepackage{amsfonts}
\usepackage{graphicx}
\usepackage[numbers, sort&compress]{natbib}
\usepackage{cite}
\usepackage{enumerate}
\usepackage{parskip}
\usepackage{cleveref}
\usepackage[affil-it]{authblk}

\newtheorem{theorem}{Theorem}
\newtheorem{lemma}{Lemma}

\newtheorem{definition}{Definition}

\newtheorem{corollary}{Corollary}

\begin{document}
\def\R{{\mathbb R}}
\def\Z{{\mathbb Z}}
\def\N{{\mathbb N}}
\def\S{{\mathbb{S}^{d-1}}}
\def\wh{\widehat}
\def\C{{\mathbb C}}
\def\M{{\mathbb M}}
\def\l({\left(}
\def\r){\right)}
\def\lV{\left\Vert}
\def\lv{\left\vert}
\def\rV{\right\Vert}
\def\rv{\right\vert}

\title{Strictly positive definite kernels on compact Riemannian manifolds} 

%\title{Strictly positive definite kernels on the $2$-sphere}

% Authors: full names plus addresses.
\author{Jean Carlo Guella,\\
Unicamp, Institute of Mathematics, Statistics and Scientific Computing, Campinas, Brazil
\and Janin J\"ager \thanks{ Corresponding author. Email:
janin.jaeger@ku.de}\\ 
KU Eichstätt-Ingolstadt, Mathematisches Institute for machine learning and data science, Ingolstadt, Germany   }
%\cortext[cor1]{Corresponding author}
%\date{\today}

\maketitle

\begin{abstract}{The paper studies strictly positive definite kernels on compact Riemannian manifolds. We state new conditions to ensure strict positive definiteness for general kernels and kernels with certain convolutional structure. We also state conditions for such kernels on product manifolds. As an example conditions for products of two-point homogeneous spaces are presented. \\}
{\textit{Keywords:} strictly positive definite kernels, covariance functions, two-point homogeneous manifolds,  }
\end{abstract}

\section{Introduction}
During the last five years, there has been a tremendous number of
publications stating new results on positive definite kernels on
spheres, see for example  \citep{ Gneiting2013, Hubbert2015} and reference therein  and a smaller number studying other manifolds \citep{Barbosa2016, Barbosa2017, Guella2017, Guella2016a,Guella2016b,Guella2022} including two-point homogeneous manifolds, tori and products of these. Most of the results focus on isotropic
positive definite kernels, which are kernels that only depend on the
geodesic distance of their arguments.  Isotropic kernels  are used in
approximation theory, where they are often referred to as spherical or 
radial basis functions \citep{Hubbert2015}  and are for example applied in
geostatistics \citep{Fornberg2015}. They
are also of importance in statistics where they occur as correlation
functions of Gaussian random fields  \citep{Berg2017a}. Some recent exceptions from the assumption of isotrpie are the axially-symmetric kernels studied on the sphere in \citep{Emery2019,Porcu2019}.

There are few results on kernel methods on general Riemannian manifolds, among them the discussion in \citep{Dyn1997, Narcowich1995, Dyn1999} forms the foundation of the presented results. The papers study generalised hermite interpolation instead of classical interpolation and, of course, do not include the results on strict positive definiteness derived for specific manifolds later.

This publication will characterise general (strictly) positive definite kernels on compact, connected Riemannian manifolds. We give necessary and sufficient conditions for kernels which are only required to be continuous and Hermitian.  The discussion on this abstract level aims to unify the approaches used for different manifolds and to provide a starting point for the investigation of new special cases. The  results will allow us to reproduce results achieved  earlier for isotropic and convolutional kernels as well as add new results.

We will briefly summarize necessary definitions in the first section
and then state the general results for characterisation of positive
definite and strictly positive definite kernels in Section 2. These results extend similar results derived for the special case of the sphere in \citep{Jaeger2021} and \citep{Buhmann2022}. Also in the second section, we describe the implications of assuming convolutional structure on the kernel and discuss conditions for (strict) positive definiteness for such kernels. As special case we derive a condition for kernels given on product manifolds. In the third section the results are discussed for the special case of homogeneous manifolds and specific results for the examples of products of two-point homogeneous spaces, which extend the results known for isotropic kernels, are given. 

\subsection{Definitions and notation}

We study approximation on a $d$-dimensional compact Riemannian manifold $\M$.
The kernels we aim to study can be used to solve interpolation problems on the manifolds. To be precise, we assume a finite set of distinct data sites $\Xi\subset \M$ and  values $f(\xi) \in \C$, $\xi \in \Xi$, of a possibly elsewhere unknown function $f$ on the manifold are given.

The approximant is taking the form 
\begin{equation}\label{eq:Interpolant}  s_f(\zeta)=\sum_{\xi\in \Xi} c_{\xi} K(\zeta, \xi), \qquad c_{\xi}\in \C,\  \zeta \in \M,
\end{equation}where \mbox{$K:\M \times \M \rightarrow \C$} is the kernel.
The  problem of finding such an approximant $s_f$ satisfying
 \begin{equation}
 s_f(\xi)=f(\xi), \qquad \forall \xi \in \Xi,
 \end{equation}
 is uniquely solvable under certain conditions on $K$. 
 \begin{definition}
Let $H$ be a $\C$-Hilbert space with inner product $\langle \cdot , \cdot \rangle$. Then a linear and continuous operator $P:H\rightarrow H$ is called positive if it is self adjoint and 
$$\langle P(x),x \rangle \geq0 ,\qquad  \forall x \in H.$$
A linear and continuous operator $P:H\rightarrow H$ is called strictly positive if it is self adjoint, bounded  and 
$$\langle P(x),x\rangle >0 ,\qquad  \forall x \in H\setminus \lbrace 0\rbrace .$$
 \end{definition}
 We assume all the kernels to be Hermitian, meaning they satisfy $K(\xi,\zeta)=\overline{K(\zeta, \xi)}$ so that the positive definiteness of the kernel will ensure the solvability of the interpolation problem for arbitrary data sets.
 
\begin{definition}\label{DF:SPD}
A Hermitian kernel $K: \Omega \times \Omega \rightarrow \C$ is {\/\rm
positive definite on} $\Omega$ if and only if the matrix  $K_{\Xi}=\left\lbrace
K(\xi,\zeta))\right\rbrace_{\xi,\zeta \in \Xi}$ induces a positive operator on $\C^{\vert \Xi\vert}$ via left  multiplication for arbitrary finite sets of distinct points $\Xi\subset \Omega$, with $\langle\cdot,\cdot\rangle$ being the standard complex scalar product.

The kernel is {\/\rm strictly positive definite} if $K_{\Xi}$ induces a strictly positive operator on $\C^{\vert \Xi\vert}$ via left  multiplication  for  arbitrary finite sets of distinct points $\Xi$.
\end{definition} 
We decided to use the given definition of positive definiteness, deviating from the standard matrix based definition for two reasons. The first in our specific setting is that we can easily switch between finite and infinite-dimensional Hilbert spaces using the operator notation and the second is that we do not have the usual notational inconsistency of defining positive definite kernels using semi-definite matrices. 

We briefly summarise the necessary results from analysis on compact Riemannian manifolds. 

The Laplace operator $\triangle$ is given in local coordinates by
$$\triangle f=\l( \operatorname{det}(g_{i,j})\r)^{-1/2}\frac{\partial}{\partial x_i}\l( \l( \operatorname{det}(g_{i,j})\r)^{1/2}g^{i,j}\frac{\partial f}{\partial x_j}\r),$$
where $g_{i,j}$ denotes the metric tensor and $g^{i,j}$ the components of its inverse.

The first list of results is cited from \citep{DeVito2019} and \citep{Donnelly2006} where reference to the original literature is given.
\begin{theorem}[Sturm-Liouville decomposition] Assume that $\M$ is a compact, connected, $d$-dimensional manifold. There exists an orthonormal base $\lbrace f_{\ell}\rbrace_{\ell \in \N}$ of $L^2(\M)$ such that each function $f_{\ell}$ is smooth and 
$$\triangle f_{\ell}=\lambda_{\ell}f_{\ell}, \quad \ell\in\N,$$
where 
$$0=\lambda_1<\lambda_2\leq \lambda_3\leq \ldots\leq \lambda_{\ell} \leq \ldots,\quad \underset{\ell\rightarrow\infty}{\lim}\lambda_{\ell}=+\infty, $$
and the multiplicity of each eigenvalue is finite. Furthermore, there exist two universal constants $C>0$ and $\ell^*$ s.\,t. for all $\ell\geq \ell^*$
\begin{equation}\label{eq:estabdeigenf} \vert f_{\ell}(\xi)\vert \leq C \lambda_{\ell}^{d/4},\quad \xi \in \M.\end{equation}
Further denote by $N_{\lambda_{\ell}}$ the multiplicity of $\lambda_{\ell}$ in the list of eigenvalues, then 
\begin{equation}\label{eq:estabdeigenvalue} N_{\lambda_{\ell}} \leq C_2 \lambda_{\ell}^{d/2}.\end{equation}
Also the vector space $\operatorname{span}\lbrace f_{\ell}\vert \ell\in\N \rbrace$ is dense in $C^{\ell}(\M)$ for all $\ell\in \N$.
\end{theorem}

Since we will work with kernels which are defined on the product manifold $\M\times \M$ we note:

For $\M$ as in the last theorem, the manifold $\M\times\M$ equipped with the product metric allows a decomposition of all functions $L^2(\M\times\M)$ using the orthonormal basis functions $f_{\ell}(\xi)\overline{f}_{\ell'}(\zeta)$ for $(\xi,\zeta)\in \M\times\M$ of eigenfunctions of the Laplacian on $\M\times\M$, the eigenvalue corresponding to this basis function is $\l(\lambda_{\ell}+\lambda_{\ell'}\r)$.

We assume throughout this paper that $K$ is  square-integrable and can be represented as
 \begin{equation}\label{eq:Kernform} K(\xi, \zeta)=\sum_{\ell,\ell'=1}^{\infty}a_{\ell,\ell'}f_{\ell}(\xi)\overline{f_{\ell'}(\zeta)},\qquad \forall \xi,\zeta \in \M ,
 \end{equation}
where the $f_{\ell}$ form an orthonormal basis of the eigenfunctions of the Laplace-Beltrami operator on $\M$ and $a_{\ell,\ell'}\in \C$. We further assume the kernel to be Hermitian, which corresponds to $a_{\ell,\ell'}=\overline{a}_{\ell',\ell}$ for all $\ell,\ell'\in \N$.

It is well known and used in the characterisation above that every function in $L^2(\M)$ can be represented as an expansion of the form 
\begin{equation}\label{eq:Serrep}
g(\xi)= \sum_{\ell=1}^{\infty}\hat{g}_{\ell}f_{\ell}(\xi),\qquad \text{ with } \hat{g}_{\ell}=\int_{\M}g(\xi)\overline{f_{\ell}(\xi)}\,d \mu( \xi),
\end{equation}
where $d\mu$ is $\det(g_{i,j})^{1/2}dx$.
For this expansion the Parseval equation holds: 
$$\Vert g\Vert_{L^2(\M)}^2=\sum_{\ell=1}^{\infty} \vert \hat{g}_{\ell}\vert ^2.$$

Also we will need the definition of the Sobolev spaces $H^r(\M)$ to be the subspace of $L^2(\M)$ with finite norm
$$ \Vert g \Vert_{H^r(\M)}:= \left( \sum_{\ell=1}^{\infty} (1+\lambda_{\ell})^{r} \vert \hat{g}_{\ell}\vert^2 \right)^{1/2},$$
where $\lambda_{\ell}$ is the eigenvalue corresponding to the eigenfunction $f_{\ell}$. We note that in this case $H^r(\M \times \M)$ is a Sobolev space with norm  
$$ \Vert K \Vert_{H^r(\M\times \M)}:= \left( \sum_{\ell,\ell'=1}^{\infty} (1+\lambda_{\ell}+\lambda_{\ell'})^{r} \vert a_{\ell,\ell'}\vert^2\right)^{1/2}.$$

We will also cite the continuous embedding lemma from Theorem 5 of \citep{DeVito2019}.
\begin{theorem}
Let $\M$ be a compact manifold. For any $r>d/2$, $H^r(\M)\subset C(\M)$. 
\end{theorem}
We further cite, from the same paper (Proposition 2), the following result on the Sobolev spaces.
\begin{lemma} Let $\M$ be compact. Given $s\in[0,\infty)$, the Sobolev space $H^s(\M)$ is a reproducing kernel Hilbert space if and only if for all $\xi \in\M$ one has
\begin{equation}\label{eq:rKHcond}
\sum_{\ell=0}^{\infty}(1+\lambda_{\ell})^{-s/2}\vert f_{\ell}(\xi)\vert^2< \infty.
\end{equation}
\end{lemma}

\section{Characterisation of (strictly) positive definite kernels}
We start by characterising strict positive definiteness of a kernel using properties of the coefficients $a_{\ell,\ell'}$. The first two results we state are closely connected to the results of Dyn, Narcowich and Ward \citep{Dyn1997} which used distributions in certain Sobolev spaces instead of sequences to describe positive definiteness.  The remaining results of this chapter generalise result proven for the two-sphere and axially-symmetric kernels in \citep{Jaeger2021} and \citep{Buhmann2022} by the second author together with Prof. Buhamnn. We give full proofs for this more general case for completeness and readability, even though many of the proof techniques are similar.  We first define the vector space  consisting of complex sequences 
%$$Y:=\left \lbrace \left( \sum_{\xi\in \Xi}c_{\xi} F_{\ell}(\xi)\right )_{\ell=1}^{\infty}  \bigg\vert  \Xi \subset \S , \vert \Xi \vert < \infty, c \in \C^{\vert \Xi \vert} \right \rbrace \subset \C^{\infty}.$$
$$Y:=\operatorname{span}\left \lbrace \left( f_{\ell}(\xi)\right )_{\ell=1}^{\infty}  \big\vert \ \xi \in \M \right \rbrace.$$

If  $K$ is continuous and therefore $K\in L^2(\M \times \M)$  the series \eqref{eq:Kernform} will converge for all $\xi,\zeta  \in \M $. Another consequence is that  
$$\int_{\M} \int_{\M}(K(\xi,\zeta))^2\,d\mu(\xi) \,d \mu( \zeta) < \infty, $$
which ensures $\sum_{\ell, \ell'=1}^{\infty} \vert a_{\ell,\ell'}\vert^2 < \infty$ because of the orthonormality and thereby makes the linear operator 
\begin{equation}\label{eq:DefOp}
 A(y):=\left( \sum_{\ell=1}^{\infty}  a_{\ell,\ell'}y_{\ell}\right)_{\ell'=1}^{\infty},\qquad y\in\ell^2(\N) \text{ or } y\in Y,
\end{equation}
bounded on $\ell^2(\N) $, by which we denote the set of square-summable complex-valued sequences.
% Since $K$ is Hermitian $\overline{a_{\ell,\ell'}}=a_{\ell',\ell}$ ensures that $A$ is a normal operator on $\ell^2$  and will therefore be invertible (on $\ell^2$) if it is bounded from below.
 We define the inner product 
 $$\langle x, y \rangle:=\sum_{\ell=1}^{\infty}x_{\ell}\overline{y_{\ell}}$$
 and will use it for $x,y\in \ell^2(\N)$ as well as for $x,y\in Y$. 
Even though we do not know whether  $A(y)$ is in $\ell^2(\N)$ for elements $y\in Y$, we know that any $y\in Y$ can be represented as $y= \left( \sum_{\xi\in \Xi}c_{\xi} f_{\ell}(\xi)\right )_{\ell=1}^{\infty}$  and therefore 
$$\langle A(y),y \rangle=\sum_{\ell=1}^{\infty}a_{\ell,\ell'}\sum_{\xi\in \Xi}c_{\xi} f_{\ell}(\xi)\overline{\sum_{\zeta\in \Xi}c_{\zeta} f_{\ell}(\zeta)}=\sum_{\xi,\zeta\in\Xi}c_{\xi}\overline{c_{\zeta}}K(\xi,\zeta)<\infty$$ 
where we are allowed to exchange the order of the summation because $K$ is continuous.

\begin{lemma}\label{LE:CharPosDef} Let $K: \M \times \M \rightarrow \C$ be a continuous Hermitian kernel given in the form \eqref{eq:Kernform}. Then $K$ is positive definite if and only if $\langle A(y),y\rangle\geq 0$ for all $y\in Y$, where $A$ is as in \eqref{eq:DefOp}. 
\end{lemma}
\begin{proof} %The proof is identical to the one given for the special case of the $2$-sphere in \citep{Jaeger2021} Lemma 2 and therefore omitted.
Let   $\langle A(y),y\rangle\geq 0$ for all  $y \in Y$, then for arbitrary sets of distinct points $\Xi \subset \M$ and $c\in \C^{\vert \Xi \vert}$ 
\begin{align*}
\sum_{\xi \in \Xi} \sum_{\zeta \in \Xi}c_{\xi} \overline{c_{\zeta}} K(\xi, \zeta)&= \sum_{\xi \in \Xi} \sum_{\zeta\in \Xi}c_{\xi} \overline{c_{\zeta} }\sum_{\ell,\ell'=1}^{\infty}a_{\ell,\ell'}f_{\ell}(\xi)\overline{f_{\ell'}(\zeta)} \\
%& = \sum_{\ell,\ell'=1}^{\infty}a_{\ell,\ell'} \sum_{\xi \in \Xi}c_{\xi}f_{\ell}(\xi)\sum_{\zeta\in \Xi} \overline{c_{\zeta}f_{\ell'}(\zeta)}  \\
&=\sum_{\ell,\ell'=1}^{\infty}y_{\ell} a_{\ell, \ell'}\overline{y_{\ell'}}=\langle A(y),y\rangle \geq 0, \\
&\text{with }   y =\left(\sum_{\xi \in \Xi}c_{\xi}f_{1}(\xi)\right)_{\ell=1}^{\infty} \in Y.
\end{align*}
The order of summation in the second line can be exchanged because the series converge for all $\xi$ since $K$ is continuous.

To prove the opposite direction we assume there exists $y \in Y$ with $\langle A(y),y \rangle <0.$
Then there exists a finite set of points $\Xi \subset \M$ such that 
$y=\left(\sum_{\xi \in \Xi}c_{\xi}f_{1}(\xi) \right)_{\ell=1}^{\infty}$  
and
$$ \sum_{\xi \in \Xi} \sum_{\zeta\in \Xi}c_{\xi} \overline{c_{\zeta} }K(\xi, \zeta) =\sum_{\ell=1}^{\infty} \sum_{\ell'=1}^{\infty}y_{\ell}a_{\ell, \ell'}\overline{y_{\ell'}}< 0.$$
\end{proof}

We can now state a characterisation of positive definiteness, which is more convenient since it only requires properties of $A$ as an operator on the Hilbert space $\ell^2(\N)$.  % in the proof of Lemma 2.
\begin{corollary}\label{COR:PosDefl2}
Let $K: \M \times \M \rightarrow \C$ be a continuous Hermitian kernel given in the form \eqref{eq:Kernform}. Then $K$ is positive definite if and only if $A$ is a positive operator on $\ell^2(\N)$, where $A$ is as in \eqref{eq:DefOp}. 
\end{corollary}
\begin{proof}
The result is easily deduced from Theorem 2.1 of \citep{Dyn1997}, by choosing $s=0$.
\end{proof}

To solve the interpolation problem introduced in Section 1 positive definiteness is not sufficient. We have to study the stricter assumption of strict positive definiteness. Applying the notation used in \Cref{LE:CharPosDef} we can state:

\begin{theorem}\label{THM:SPDCHar}
Let $K: \M \times \M \rightarrow \C$ be a Hermitian continuous kernel given in the form \eqref{eq:Kernform}. Then $K$ is strictly positive definite if and only if $\langle A(y),y\rangle >0$ for all  $y \in  Y\setminus \lbrace 0 \rbrace $, where $A$ is as in \eqref{eq:DefOp}. 
\end{theorem}
\begin{proof}
For the first direction we assume that $K$ is strictly positive definite. Then it will also be positive definite and thus $\langle A(y),y\rangle \geq0$ for all $y \in Y$ according to \Cref{LE:CharPosDef}. Further for all elements $y\in Y \setminus \lbrace 0\rbrace$ there exists  a set $\Xi$ of distinct points and coefficients $c_{\xi}$ not all zero s.t. $y =\left(\sum_{\xi \in \Xi}c_{\xi}f_{\ell}(\xi)\right)_{\ell=1}^{\infty}$.
Then again 
\begin{align*}
\sum_{\xi \in \Xi} \sum_{\zeta \in \Xi}c_{\xi} \overline{c_{\zeta}} K(\xi, \zeta)=\langle A(y),y\rangle >0,
\end{align*}
as long as $y\neq0$. Completing the first direction.

Now assume $\langle A( y) ,y\rangle > 0$ for all  $y \in  Y\setminus \lbrace 0 \rbrace$. Then $K$ is positive definite following from \Cref{LE:CharPosDef}. Further we show that for a finite non empty set of distinct points $\Xi\subset \M$
$$y=\left(\sum_{\xi \in \Xi}c_{\xi}f_{\ell}(\xi) \right)_{\ell=1}^{\infty}= 0\in \C^{\infty}$$  if and only if  $c=0\in \C^{\vert \Xi \vert}$.
This is a consequence of the $f_{\ell}$ forming a basis of $L^2(\M)$ and 
$$\sum_{\xi \in \Xi}c_{\xi}f_{\ell}(\xi)=0,\qquad  \forall \ell \in \N,$$
would implicate   $\sum_{\xi \in \Xi}c_{\xi}f(\xi)=0$ for all $f\in L^2(\M)$, but the evaluation functionals on $L^2(\M)$ are linearly independent. This proves that 
$$\sum_{\xi \in \Xi} \sum_{\zeta \in \Xi}c_{\xi} \overline{c_{\zeta}} K(\xi, \zeta)=\langle A(y),y \rangle>0,$$
 as long as $0\neq c\in \C^{\vert \Xi \vert}$.
\end{proof}

We will find that strict positive definiteness on $\ell^2(\N)$ is not necessary for strict positive definiteness of the kernel.
%We want to highlight two other properties of $A$ which induce positive definiteness of the kernel.
To be able to work without the assumptions on smoothness properties of the kernel we define the summability property on the coefficient matrix as follows. 
Let 
$$\tilde{A}^s=\l(\tilde{a}_{\ell,\ell'}\r)_{\ell,\ell'=0}^{\infty},\text{ with }\tilde{a}_{\ell,\ell'}^s=\l(1+\lambda_{\ell}\r)^{s}\vert a_{\ell,\ell'} \vert \l(1+\lambda_{\ell'}\r)^{s}.$$ We note the following implication for the smoothness of a kernel:
\begin{lemma}
Let $K$ be a kernel of the form \eqref{eq:Kernform}. If
\begin{equation}\label{eq:Abssumprop}
\sum_{\ell,\ell'=0}^\infty \vert \tilde{a}^s_{\ell,\ell'}\vert<\infty,\quad s>0, 
\end{equation}
then $K\in H^{2s} (\M\times \M)$.
\end{lemma}
\begin{proof}From \eqref{eq:Abssumprop} we deduce
\begin{align*}
\infty&>\sum_{\ell,\ell'=0}^\infty \vert (1+\lambda_{\ell})^s a_{\ell,\ell'}(1+\lambda_{\ell'})^s\vert\\
&\geq  \sum_{\ell,\ell'=0}^\infty \vert (1+\lambda_{\ell}+\lambda_{\ell'})^s a_{\ell,\ell'}\vert.
\end{align*}
From which 
\[\sum_{\ell,\ell'=0}^\infty  (1+\lambda_{\ell}+\lambda_{\ell'})^{2s} \vert a_{\ell,\ell'}\vert^2\leq \infty\]
follows, proving $K\in H^{2s} (\M\times \M)$.
\end{proof}

\begin{lemma}\label{Lemma:UniqueRep}
Let $K$ be a kernel of the form \eqref{eq:Kernform}. If there exists an $s\geq d/4$ with $\sum_{\ell,\ell'=1}^\infty \vert \tilde{a}^s_{\ell,\ell'}\vert <\infty$, 
then, for each combination of  $\ell,\ell'\in \N,$ the coefficients $a_{\ell,\ell'}$ can be uniquely determined through
\begin{equation}\label{eq:FouriercoefKer}
a_{\ell,\ell'}=\int_{\M}\int_{\M}K\left(\xi,\zeta\right)\overline{f_{\ell}}\left(\xi\right) {f_{\ell'}}\left(\zeta\right)\, d\mu\left(\xi\right)\,d\mu\left(\zeta\right).
\end{equation}
\end{lemma}
\begin{proof} The result follows by inserting the kernel representation \eqref{eq:Kernform} into  \eqref{eq:FouriercoefKer} and then using the estimate \eqref{eq:estabdeigenf} to exchange the order of summation and integration since it implies absolute summability.
\end{proof}

We note that a kernel as in the last lemma will always be continuous because of the smoothness of the basis functions and the absolute convergence of the series representation.

\begin{lemma}\label{LE:SubSemiPos}
Let $K: \M \times \M \rightarrow \C$ be a kernel satisfying \eqref{eq:Abssumprop}, for $s\geq d/4$ with $K(\xi,\zeta)=\overline{K(\zeta,\xi)}$. Then $K$ is positive definite if and only if the matrix $A_k:=(a_{\ell,\ell'})_{\ell,\ell'=1 }^k$ is positive semi-definite for all $k\in \N$.
\end{lemma}
\begin{proof}
First assume $A_k$ is positive definite for all $k$ and $K$ is as above. For any set of distinct points $\Xi \subset \M$ and any $c\in \C^{\vert \Xi \vert}$, and $y$ as in the proof of \Cref{THM:SPDCHar},
\begin{align*}
\sum_{\xi \in \Xi} \sum_{\zeta \in \Xi}c_{\xi} \overline{c_{\zeta}} K(\xi, \zeta)
&=\sum_{\ell,\ell'=1}^{\infty}y_{\ell} a_{\ell, \ell'}\overline{y_{\ell'}}=\underset{k \rightarrow \infty}{\lim} \sum_{\ell,\ell'=1}^{k}y_{\ell} a_{\ell, \ell'}\overline{y_{\ell'}}\geq 0.
\end{align*}
The other direction follows by contradiction, assuming $A_k$ is not positive semi-definite contradicts the positive definiteness of the operator $A$ on $\ell^2$, which has been shown to be necessary for positive definiteness of the kernel in \Cref{COR:PosDefl2}.
\end{proof}

\begin{theorem}
Let $K$ be an Hermitian kernel of the form \eqref{eq:Kernform} satisfying \eqref{eq:Abssumprop} for $s\geq d/4$.
The kernel is strictly positive definite if  $\tilde{A}^{d/4}$ satisfies the uniform strict diagonal dominance property:
$$\sum_{\ell'\neq \ell} \left\vert \tilde{a}^{d/4}_{\ell,\ell'}\right\vert < \sigma \left\vert \tilde{a}^{d/4}_{\ell,\ell}\right\vert,\quad \forall \ell \in \N, $$
where $0< \sigma <1$ is independent of $\ell$. 
\end{theorem}
\begin{proof}
We find that all submatrices of $\tilde{A}^{d/4}$  are strictly diagonally dominant with non-zero diagonal entries and thereby positive definite. We can define the Cholesky decomposition of $\tilde{A}^{d/4}$ as limit of the finite Cholesky decompositions of the submatrices. We denote by $\tilde{A}^{d/4}=LL^*$ the infinite Cholesky decomposition matrix which exists since $\tilde{A}^{d/4}$ is positive definite and bounded on $\ell^{\infty}(\N)$ as a consequence of \eqref{eq:Abssumprop}, $L^*$ is the complex conjugate of $L$ and transposed.

Inserting the Cholesky-decomposition into the quadratic form 
\begin{align}\label{eq:CholeskyKern}
\sum_{\xi \in \Xi} \sum_{\zeta \in \Xi}c_{\xi} \overline{c_{\zeta}} K\left(\xi, \zeta\right)
&=\lV L\mathbf{y}\rV_2^2,
\end{align}
where $\tilde{\mathbf{y}}=\left(\frac{1}{(1+\lambda_{\ell})^{d/4}}y_{\ell}\right)_{\ell=1}^{\infty}$ is a bounded sequence as a result of the estimate of the  eigenfunctions of the Laplace Beltrami operator in \eqref{eq:estabdeigenf} and we can rearrange the sum because of the absolute summability.

Assume there exists an element of the space of bounded sequences $\ell^{\infty}(\N)$ for which $L^*x=0$, then $L\left(L^*x\right)=\tilde{A}^{d/4}x=0$, where we can exchange the order of multiplication because of the triangular structure of $L$ and because $L^*x$ is element wise finite for all bounded $x$, as a result of $\lV L^*x \rV^2 =\overline{x}^T\tilde{A}^{d/4}x <\infty$ which again is a consequence of \eqref{eq:Abssumprop}.

Now it is sufficient to prove that there exists no eigenvector in the space of bounded sequences for which $\tilde{A}^{d/4}x=0.$
Therefore we can employ the results for non singularity of infinite matrices given in \citep{Shivakumar1987} Theorem 1b. We briefly review their arguments for the convenience of the reader.

From the strict diagonal dominance we can deduce that $\tilde{a}^{d/4}_{\ell,\ell}$ is non zero for all $\ell\in\N$. Assume there exists $x\in\ell^{\infty}(\N)$ with
$\tilde{A}^{d/4}x=0$ and $\lV x\rV_{\infty}=1$. Thereby 
\begin{align*} \sum_{\ell'=1}^{\infty} \tilde{a}^{d/4}_{\ell,\ell'}x_{\ell'}&= 0,\quad  \forall \ell \in \N \ 
\Leftrightarrow \ \sum_{\ell'\neq \ell}^{\infty} \tilde{a}^{d/4}_{\ell,\ell'}x_{\ell'}= -\tilde{a}^{d/4}_{\ell,\ell}x_{\ell}, \quad \forall \ell \in \N.
\end{align*}
Since $x$ has $\ell^{\infty}$ norm one, there exists a value $\ell\in \N$, for which $\vert x_{\ell} \vert>\sigma$ and
$$\sum_{\ell'\neq \ell}^{\infty}\left\vert  \tilde{a}^{d/4}_{\ell,\ell'}\right\vert \geq  \sum_{\ell'\neq \ell}^{\infty}\left \vert  \tilde{a}^{d/4}_{\ell,\ell'} x_{\ell'}\right \vert \geq  \left\vert \tilde{a}^{d/4}_{\ell,\ell}\right \vert \sigma.$$
This contradicts the assumption and the only possible choice for which \eqref{eq:CholeskyKern} is equal to zero is $\tilde{\textbf{y}}=0$. From the definition of this sequence and the independence of the basis functions we know that this is only possible if $c_{\xi}=0$ for all $\xi \in \Xi$.
\end{proof}
The above condition on the matrix is denoted as strict uniform diagonal dominance and it is a stronger assumption than diagonal dominance. In the next theorem we show that diagonal dominance of $\tilde{A}^s$ is sufficient if we add a stronger assumption on $s$.
\begin{theorem} Let $s\geq d/4$ be such that  $\sum_{\ell=0}(1+\lambda_{\ell})^{-2s}\vert f_{\ell}(\xi)\vert^2< \infty$ and the Hermitian kernel of the form \eqref{eq:Kernform} satisfies
$$\sum_{\ell,\ell'=0}^{\infty} \left\vert \tilde{a}^{s}_{\ell,\ell'} \right\vert<\infty,$$
then it is sufficient for the strict positive definiteness of the kernel that $\tilde{A}^s$ is strictly diagonally dominant.
\end{theorem}
\begin{proof}
As in the proof of the last theorem we denote by $\tilde{A}^{s}=LL^*$ the infinite Cholesky decomposition matrix which exists since  the submatrices  $\tilde{A}^{s}_k$ are positive definite and  the matrix is bounded on $\ell^{\infty}(\N)$ as a consequence of \eqref{eq:Abssumprop} and $s\geq d/4$.

Again inserting the Cholesky-decomposition into the quadratic form yields
\begin{align}
\sum_{\xi \in \Xi} \sum_{\zeta \in \Xi}c_{\xi} \overline{c_{\zeta}} K\left(\xi, \zeta\right)
&=\lV L{\mathbf{y}}\rV_2^2,
\end{align}
where $\tilde{\mathbf{y}}=\left(\frac{1}{(1+\lambda_{\ell})^{s}}y_{\ell}\right)_{\ell=1}^{\infty}$ is now an $\ell^2(\N)$ sequence because of the condition
 on $s$. With the same arguments as in the last proof, we find that if for some $x\in \ell^2(\N)$  $L^*x=0$, then $L\left(L^*x\right)=\tilde{A}^{s}x=0$.

Thereby it is sufficient to show that there exists no eigenvector in the space $\ell^2(\N)$ for which $\tilde{A}^{s}x=0.$ 

The result now follows by applying, \citep{Farid1995} Theorem 5.1. We include the adapted proof for our setting for completeness.
Suppose that $0$ is an eigenvalue of $\tilde{A}^s$ relative to the space of null sequences, which includes $\ell^2(\N)$, then there exists a null sequence $x$ not equal to zero for which 
$\tilde{A}^{s}x=0$.
Since $\vert x_i\vert\rightarrow 0$ as $i\rightarrow \infty$ the set $\lbrace \vert x_i\vert\ : \ i\in \N\rbrace$ has a maximum, let $k$ be an integer where $\vert x_k\vert$ is maximal.
%Thus there is a positive integer $k$ s. t. $\vert x_k\vert=\max\langle{\vert x_i \vert \, : \i\in\N\rangle$.
%Also the set $K:=\langle i \in \N\, :\, \vert x_i\vert=\vert x_k\vert \rangle$ is finite. Set $N\in\N$ to be the largest integer in $K$ and $M\in\N$ to be the smallest integer in $K$. 
We find that 
$$ \vert \tilde{a}^{s}_{k,k}\vert \vert x_k\vert \leq \sum_{\ell \neq k}\vert \tilde{a}^s_{k,\ell}\vert \vert x_{\ell}\vert \leq \sum_{\ell \neq k} \vert \tilde{a}^s_{k,\ell}\vert \vert x_k\vert.$$
Implying $$\vert \tilde{a}^{s}_{k,k}\vert\leq\sum_{\ell \neq k} \vert \tilde{a}^s_{k,\ell}\vert,$$
which contradicts the strict diagonal dominance.

%Also from the above estimate we deduce that for any $j \in \N$  
% with $\tilde{a}^s_{k,j}\neq 0$ the equality $\vert x_j\vert=\vert x_k\vert$ holds. Thereby we find that $\tilde{a}^s_{i,j}=0$ for all $i\in K$ and $j\geq N+1$. Hence $\sum_{j=1}^N \tilde{a}^s_{n,j}x_j=0$. But since $\tilde{A}^s$ has positive diagonal elements $\tilde{a}^s_{N,N}$, because of the strict positive definiteness of the submatrices, and $x_N$ is non zero there need to be a $t<N$ s.t. $\tilde{a}^s_{N,t}\neq 0$ which deduces $\vert x_t\vert=\vert x_k\vert$. This implies $M<N$. Therefore $K=\left\lbrace \tau_1,\ldots,\tau_{\vert K\vert}\right\rbrace$ with $M=\tau_1<\cdots<\tau_{\vertK\vert} =N$.
%As a consequence in the truncated matrix $\tilde{A}^s_N$ all entries $\tilde{a}^s_{j,k}$, where $j\in K$ but $k\notin K$ are zero. We define the vector $\tilde{x}\in\C^{N}$ by setting $\tilde{x}_{j}=x_{j}$, for $j \in K$ otherwise $\tilde{x}_j=0$ . Then $\tilde{A}^s_N\tilde{x}=0$ which contradicts the non-singularity of this matrix.
\end{proof}

%\section{Kernels with general block structure}

\subsection{Kernels of convolutional form}

For a special class of kernels we can simplify the condition of strict positive definiteness. This class was also studied in \citep{Dyn1997} and will allow for severe simplification.
Let $K$ be a kernel of the form 
\begin{equation}\label{eq:Conkerform}
K(\xi, \zeta)=\sum_{\ell=1}^{\infty}a_{\ell}f_{\ell}(\xi)\overline{f_{\ell}(\zeta)},\qquad \forall \xi,\zeta \in \M .
\end{equation}
We will refer to these kernels as kernels of convolutional form and define the set 
$$\mathcal{F}:=\left\lbrace \ell\in \N : a_{\ell}>0\right\rbrace.$$
\begin{theorem}\label{THMSerKerSPD}
A continuous kernel of the form \cref{eq:Conkerform} is positive definite if and only if $a_{\ell}\geq 0$ for all $\ell\in\N$. It is strictly positive definite if and only if it is positive definite and for any finite set of distinct points $\Xi$, $\sum_{\xi \in \Xi}c_{\xi} f_{\ell}(\xi)= 0$, for all $\ell \in \mathcal{F}$ implies $c_{\xi}=0$ for all $\xi \in \Xi$. In this case we say $\mathcal{F}$ induces strict positive definiteness on $\M$.
\end{theorem}
\begin{proof}
The first part of the theorem follows directly from \Cref{LE:CharPosDef} and was already stated in \citep{Dyn1997}. 
We prove the second result by contradiction and assume that $K$ is continuous positive definite but not strictly positive definite.  If $K$ is not strictly positive definite there exists a nonempty set of distinct point $\Xi$ and coefficients $c_{\xi}$ not all zero with  
$$ \sum_{\xi,\zeta \in \Xi}c_{\xi} \overline{c_{\zeta}} K\left(\xi, \zeta\right)=0. $$
This is equivalent to 
\begin{equation}\label{eq:Formaxker3}
\sum_{\ell=1}^{\infty}a_{\ell} y_{\ell} \overline{y_{\ell}}=0,
\end{equation}
where $y_{\ell}=\sum_{\xi\in\Xi} c_{\xi} f_{\ell}\left(\xi\right)$ and the sums are interchangeable because of the continuity.

Since we know that all the summands are non negative since the $a_{\ell}$ are non-negative, the overall sum can only be zero if all summands are. For the indices $\ell\in \mathcal{F}$ this implies  $y_{\ell}=0$.

But according to the condition of the lemma this implies $c_\xi=0$ for all $\xi\in\Xi$, in contradiction to the assumption of the proof.

\end{proof}

\subsection{Structure on product manifolds}
For products of manifolds we can generalise a sufficient condition proven for  spheres in  Theorem 2.5, \citep{Guella2016b}.
We investigate strict positive definiteness on manifold of the product structure $\M \times \mathbb H$, where both $\M$ and $\mathbb H$ are connected, compact Riemannian manifolds. We will again focus on convolutional kernels.
We can now define an orthonormal basis of the eigenfunctions of the Laplace-Beltrami operator on the product manifold, according to \citep{Dyn1997}, this takes the form 
\begin{equation}\label{prodbasis}f_{\ell}(\xi) g_{j}(\zeta),\quad  (\xi,\zeta)\in \M \times \mathbb H,
\end{equation}
where $f_{\ell}$ is a basis as described in the last section for $L^2(\mathbb M)$ and $g_{j}$ is a similar orthonormal basis of $L^2(\mathbb H)$. 

These kernels of convolutional form can be represented in a series expansion of the form \eqref{eq:Conkerform}. To be specific such a continuous and Hermitian kernel has a representation of the form
\begin{equation}\label{eq:ConkernProdform}
K((\xi,\zeta),(\xi',\zeta'))=\sum_{(\ell,j)\in \N^2}a_{(\ell,j)}f_{\ell}(\xi)g_{j}(\zeta)\overline{f_{\ell}(\zeta')}\overline{g_{j}(\xi')},
\end{equation}with $  a_{(\ell,j)}\in\R$ and $ (\xi,\zeta)$,  $(\xi',\zeta')\in \M\times \mathbb{H}$.
We will now derive sufficient conditions for the strict positive definiteness of such kernels.

\begin{lemma}\label{THMSerKerSPD}
A kernel of the form \eqref{eq:ConkernProdform} is positive definite if $a_{(\ell,j)}\geq 0$ for all $(\ell,j)\in\N^2$. It is strictly positive definite if  the set $$\mathcal{G}:=\lbrace \ell :\mathcal{G}_{\ell} \text{ induces strict  positive definiteness on } \mathbb H \rbrace$$ induces strict positive definiteness on $\M$, with  $\mathcal{G}_{\ell}=\lbrace j\in \N : (\ell,j)\in \mathcal{F}\rbrace$.
\end{lemma}
\begin{proof} The first part follows directly from \Cref{THMSerKerSPD}.
Now assume that $$\mathcal{G}:=\lbrace \ell :\mathcal{G}_{\ell} \text{ induces strict  positive definiteness on } \mathbb H \rbrace$$ induces strict positive definiteness on $\M$. According to \Cref{THMSerKerSPD} adapted to the product setting, we prove that for any finite subset of distinct data sites in $\mathbb M \times \mathbb H$, denoted by $\Xi$
assuming 
\begin{equation}
\sum_{(\xi,\zeta)\in\Xi}c_{(\xi,\zeta)}f_{\ell}(\xi)g_{j}(\zeta)=0,\quad \forall \ (\ell,j)\in \mathcal{F},
\end{equation}
will imply the coefficients to be all zero.
The last display is adapted by summing only over the distinct values of $\zeta$ first. Therefore we define $\Xi^1:\lbrace \zeta \in \mathbb H: \exists (\xi,\zeta)\in \Xi \rbrace$ and $\Xi_{\zeta}:=\lbrace \xi: (\xi,\zeta)\in \Xi\rbrace$  and find 
\begin{equation*}
\sum_{\zeta\in\Xi^1}\left( \sum_{\xi \in \Xi_{\zeta}}c_{(\xi,\zeta)}f_{\ell}(\xi)\right)g_j(\zeta)=0,\quad \forall \ (\ell,j)\in \mathcal{F}.
\end{equation*}
The outer sum needs to be zero especially for all pairs $(\ell,j)\in\mathcal F$ with $\ell \in \mathcal{G}$. Fixing $\ell \in \mathcal{G}$ implies that $\mathcal{G}_{\ell}$ induces strict positive definiteness on $\mathbb H$ and since all $\zeta \in \Xi^1$ are distinct
\begin{equation*}
\sum_{\zeta\in\Xi^1}\left( \sum_{\xi \in \Xi_{\zeta}}c_{(\xi,\zeta)}f_{\ell}(\xi)\right)g_{j}(\zeta)=0,\quad \forall j\in \mathcal{G}_{\ell},
\end{equation*}
implies 
\begin{equation*}
 \sum_{\xi \in \Xi_{\zeta}}c_{(\xi,\zeta)}f_{\ell}(\xi)=0,\quad  \forall \ell \in \mathcal{G}.
 \end{equation*}
 The last equation implies $c_{(\xi,\zeta)}=0$ for all $(\xi,\zeta)\in \Xi$ since all $\xi \in \Xi_{\zeta}$ are distinct and $\mathcal G$ induces strict positive definiteness on $\mathbb M$. 
\end{proof}
\section{Strict positive definiteness on homogeneous manifolds}
 We add the additional assumption of the manifold being homogeneous and discuss the resulting changes to the general case. We denote by $G$ the group of isometries of the manifold, which acts transitive on it since the manifold is homogeneous. 
 
The eigenfunctions of the Laplacian in this setting were discussed in \citep{Evarist1975} and we summarise the additional properties using the notation of eigenfunctions and eigenvalues from the last section.
 
 From the invariance under isometries of the Laplacian we know that for any $g\in G$:
 $$ \triangle f_{\ell}(g(\xi))=\lambda_{\ell}f_{\ell}(g(\xi))$$
and $f_{\ell}\cdot g$ is an eigenfunction corresponding to  the same eigenvalue as $f_{\ell}$.

We cite the addition formula for the eigenfunctions from \citep{Evarist1975} (Lemma 3.1 and Theorem 3.2).
From now on we use an index $k$ to enumerate the distinct eigenvalues of the Laplacian, $0=\lambda_1<\lambda_2\ldots$.
For any $\xi_0\in \M$ and eigenspace $H_k$ of $\triangle$ of dimension $m_k$ corresponding to the eigenvalue $\lambda_k$,  there exists a unique function $f^k_{\xi_0}$ which satisfies
\begin{enumerate}
\item it is zonal with respect to $\xi_0$, 
\item $f^k_{\xi_0}(\xi_0)=m_k^{1/2}$,  
\item $f^k_{\xi_0}$ is orthogonal to any function in $H_k$ vanishing at $\xi_0$,
\item $\int_{\M}(f^k_{\xi_0}(\zeta))^2\,d\mu(\zeta)=1$. 
\end{enumerate}

Further for any orthonormal basis of $H_k$, $f_{1,k},\ldots, ,f_{m_k,k}$, and   for every $\xi,\zeta \in \M$,
\begin{equation}\label{eq:AdditionTheorem}
\sum_{j=1}^{m_k}f_{j,k}(\xi)\overline{f_{j,k}(\zeta)}=m_k^{1/2}f^k_{\zeta}(\xi)=m_k^{1/2}f^k_{\xi}(\zeta).
\end{equation}
%where, for any two points $\xi,\zeta\in \M$. $\phi_{\xi,\zeta}$ denotes any isometry mapping $\xi$ into $\zeta$.
% $\sum_{i=1}^n \vert f_i(\xi)\vert^2\equiv n$.

We adapt the series representation \eqref{eq:Conkerform} of the prior section in the following way:

\begin{equation}\label{eqKernHomConv}
K(\xi,\zeta)=\sum_{k=0}^{\infty}\sum_{j=1}^{m_k}d_{j,k}f_{j,k}(\xi)\overline{f_{j,k}(\zeta)},\quad  d_{j,k}\in\R,
\end{equation}
where $f_{j,k}$, $j=1,\ldots,m_k$ is an orthonormal basis of $H_k$.

We summarise a characterisation of strict positive definitneness of such kernels analogue to the cases derived for spheres.
For a positive definite kernel of the form \eqref{eqKernHomConv} we define 
	$$\mathcal{F}:=\lbrace (j,k) :\ d_{j,k}>0\rbrace,\quad
	A_k:=\lbrace j: (j,k)\in \mathcal{F} \rbrace, \quad \mathcal{N}=\lbrace k: k\in \Z_+ \wedge \exists d_{j,k} \neq 0 \rbrace.$$
	
\begin{theorem}\label{LEEquivspdGinv}
		For a  continuous p.d. kernel as in  \eqref{eqKernHomConv} the following are equivalent:
		\begin{enumerate}
			\item $K$ is strictly positive definite on $\M$.
			\item For any finite set of distinct points $\Xi$, $\sum_{\xi \in \Xi}c_{\xi} f_{j,k}(\xi)= 0$, for all $(j,k) \in \mathcal{F}$ implies $c_{\xi}=0$ for all $\xi \in \Xi$.
			\item For any finite set of distinct points $\Xi\subset\M$, $$ \sum_{\xi\in \Xi }c_{\xi} \sum_{j\in A_k}f_{j,k}(\xi)\overline{f_{j,k}(\zeta)}=0,\ \forall k \in N,\zeta \in \M,\text{ implies }c_{\xi }=0,\ \forall \xi \in \Xi.$$
		\end{enumerate}
	\end{theorem}
\begin{proof}
The equivalence of one and two was already established in Theorem 6, which only needs to be adapted to the double index setting. 
Using the addition formula of the eigenfunction \eqref{eq:AdditionTheorem} we find,
since the eigenfunctions are linearly independent, the last is zero if and only if 
$$\sum_{\xi \in \Xi}c_{\xi}f_{j,k}(\xi)=0$$ for all $j=1,\ldots,m_k$, and all $k\in \mathcal{N}$. 
\end{proof}

We further define a set of $G$ invariant kernels. The kernels have already been studied in \citep{Levesley2007} and \citep{Odell2012} but with a derivation using an embedding and resulting in a slightly different representation.  
All kernels of the form 
\begin{equation}\label{eqKernHomRad}
K(\xi,\zeta)=\sum_{k=0}^{\infty}d_{k}f^k_{\xi}(\zeta),\quad  d_{k}\in\R,
\end{equation}
are invariant under $G$.
For these kernels we define the set $$\mathcal{G}:\lbrace k: d_k>0\rbrace$$
and find as a consequence of the last theorem.
\begin{corollary}
A continuous $G$-invariant positive definite kernel of the form \eqref{eqKernHomRad} is strictly positive definite if and only if 
for any finite set of distinct points $\Xi\subset\M$, $$ \sum_{\xi\in \Xi }c_{\xi} f_{\xi}^k(\zeta)=0,\ \forall k \in \mathcal{G},\zeta \in \M,\text{ implies }c_{\xi }=0,\ \forall \xi \in \Xi.$$
We say such a set induces strict positive definiteness for $G$-invariant kernels.
\end{corollary}
From strict positive definiteness of $G$-invariant kernels, conditions for kernels of the form \eqref{eqKernHomConv} can be deduced. This is useful in applications because the class of $G$-invariant kernels is more intensely studied and in many settings simple characterizations of the sets $\mathcal{K}$ inducing strict positive definiteness exist. 
For a kernel of convolutional form we define two subset of $\N$:
\begin{equation}\label{defUL} \begin{aligned}
\mathcal{U}:=\left\lbrace k\in \N: \ \exists j\in \lbrace 1,\ldots,m_k\rbrace \text{ with }d_{j,k}\neq 0\right\rbrace,\\
\mathcal{L}:=\left\lbrace k\in \N: \ d_{j,k}> 0,\ \forall j\in \lbrace 1,\ldots,m_k\rbrace\right \rbrace.
\end{aligned}
\end{equation}
From the definition it is clear that for any kernel $\mathcal{L}\subseteq\mathcal{U}$ and equality holds for (but not only for) $G$-invariant kernels.
\begin{lemma}\label{LE:HomConNec}\begin{enumerate}
\item For  a continuous positive definite kernel given as \eqref{eqKernHomConv}  to be strictly positive definite it is necessary that $\mathcal{U}$ induces strict positive definiteness for $G$-invariant kernels on $\M$.
\item For  a continuous positive definite kernel given as \eqref{eqKernHomConv} to be strictly positive definite it is sufficient that $\mathcal{L}$ induces strict positive definiteness for $G$-invariant kernels on $\M$.
\end{enumerate}
\end{lemma}
\begin{proof}
We prove (1) by contradiction. Assume $\mathcal{U}$ does not induce strict positive definiteness for homogeneous kernels.   Then there exists a set of data sites $\xi \in \Xi$ and coefficients not all zero satisfying:
$$\sum_{\xi\in\Xi}c_{\xi}\sum_{j=1}^{m_k} f_{j,k}(\xi)\overline{f}_{j,k}(\zeta)\equiv 0,\quad  \forall k\in \mathcal{U}, \zeta \in \M. $$
Exchanging the two finite sums and applying the linear independence of the function $f_{k,j}(\zeta)$ the latter implies 
$$\sum_{\xi\in\Xi}c_{\xi} f_{j,k}(\xi)=0,\quad \forall  k\in \mathcal{U},\ j=1,\ldots, m_k.$$
The last statement contradicts the strict positive definiteness of $K$ as characterised in \eqref{THMSerKerSPD}.
The second statement follows directly form the equivalence of (1) and (2) in \Cref{LEEquivspdGinv}.
\end{proof}

%\begin{proof}
%According to \ref{THMSerKerSPD} the kernel is strictly positive definite if 
%\begin{equation}\label{eqTwoPointImpl} \sum_{\xi\in\Xi}c_{\xi} f_{k,j}(\xi)=0,\quad \forall k\in \mathcal{F}, j=1,%\ldots,m_k
%\end{equation}
%implies $c_{\xi}=0$ for all $\xi \in \Xi$. 
%Assume that all $\tilde{K}$ are strictly positive definite and  \cref{eqTwoPointImpl}. Then we can deduce 

%$$\sum_{\xi\in\Xi}c_{\xi}\sum_{j=1}^{m_k} f_{k,j}(\xi)\overline{f}_{k,j}(\zeta)\equiv 0,\quad  \forall k\in \mathcal{F}, \zeta \in \M, $$
%which together with the assumption about $\tilde{K}$ implies $c_{\xi}=0$ for all $\xi \in \Xi$ according to Lemma 2.1. of %\citep{Barbosa2016}.
%\end{proof}

For the case of the $d-$sphere the properties implied by this lemma were already described in \citep{Jaeger2021} and for the case of two-point homogeneous manifolds results are described in  \citep{Guella2022}. The later also presents an approach of constructing non-isotropic kernels having significantly less positive coefficients as the isotropic kernels. As an additional example we study products of two point homogeneous manifolds.

\subsection{Example: Products of two-point homogeneous manifolds}
The next section investigates strict positive definiteness on manifold of the product structure $\M \times \mathbb H$, where both $\M$ and $\mathbb H$ are two-point homogeneous manifolds.
Therefore they are isomorphic to one of the following five cases as proven in \citep{Wang1952},
\begin{align*}
\mathbb{S}^{d-1},\quad P^{d-1}(\R),\quad
P^{d-1}(\C),\quad P^{d-1}(H),\quad
P^{16}(Cay).
\end{align*}

With the notation of the previous section and $d(\cdot,\cdot)$ the distance between its to arguments the addition formula from  \citep{Evarist1975} reads
\begin{equation}\label{eqSummationFor}
\sum_{j=1}^{m_k}f_{k,j}(\xi)\overline{f_{k,j}(\zeta)}=c_{k}P_{\epsilon k}^{(\alpha,\beta)}\left(\cos\left(\epsilon^{-1}d(\xi,\zeta)\right)\right),\quad \xi,\zeta \in \M,
\end{equation}
where
\begin{equation*} c_k=\frac{\Gamma(\beta+1)(2k+\alpha+\beta+1)\Gamma(k+\alpha+\beta+1)}{\Gamma(\alpha+\beta+2)\Gamma(k+\beta+1)}
\end{equation*}
and throughout $P_k^{(\alpha,\beta)}$, denotes the Jacobi polynomials normalized by
\begin{equation}
P_k^{(\alpha,\beta)}(1)=\frac{\Gamma(k+\alpha+1)}{\Gamma(k+1)\Gamma(\alpha+1)}.
\end{equation}
The coefficients satisfy $\alpha=\frac{d-3}{2}$, $\beta$ takes one of the values $(d-3)/2,\,-1/2,\,0,\,1,\,3$, depending on  which of the five cases of manifolds is studied, and $\epsilon=1$ unless $\M=P^{d-1}(\R)$ in which case $\epsilon=2$.

To be specific a kernels of the from \eqref{eq:ConkernProdform} in this setting has a representation of the form
\begin{equation}\label{eqKernTwopointProdConv}
K((\xi,\zeta),(\xi',\zeta'))=\sum_{k,k'=0}^{\infty}\sum_{j=1}^{m_k}\sum_{j'=1}^{m_{k'}}a_{j,j',k,k'}f_{j,k}(\xi)g_{j',k'}(\zeta)\overline{f_{j,k}(\zeta')}\overline{g_{j',k'}(\xi')}, 
\end{equation}with $a_{j,j',k,k'}\geq 0$ for a positive definite kernel.
%From \citep{Dyn1999} we can deduce that such a kernel is positive definite if and only if $a_{j,j',k,k'}\geq0$.
We will now derive sufficient conditions for the strict positive definiteness of the kernel by employing \Cref{THMSerKerSPD} and the results existing for isotropic kernels.

For a kernel of the \eqref{eqKernTwopointProdConv} define the two subset of $\N^2$:
\begin{align*}
\mathcal{J}:=\left\lbrace (k,k')\in \N^2:\ \exists (j,j') \text{ with } a_{j,j',k,k'}> 0\right\rbrace,\\
\mathcal{F}:=\left\lbrace (k,k')\in \N^2:\ a_{k,k',j,j'}> 0\ \forall (j,j') \right \rbrace.
\end{align*}

%We prove the necessity by contradiction. Assume $\mathcal{J}$ does not induce strict positive definiteness of G-invariant kernels then there exists a set of data sites $\xi \in \Xi$ and coefficients not all zero satisfying:
%$$\sum_{(\xi,\zeta)\in\Xi}c_{(\xi,\zeta)} P_{\epsilon k}^{(\alpha,\beta)}\left(\cos\left(\epsilon^{-1}d(\xi,\xi')\right)\right)P_{\epsilon k'}^{(\alpha,\beta)}\left(\cos\left(\epsilon^{-1}d(\zeta,\zeta')\right)\right)\equiv 0 $$
%for all $(k,k'
%)\in \mathcal{J}$ and any  $(\xi',\zeta') \in \M$ according to \Cref{LEEquivspdGinv}.
%Replacing the Legendre polynomials using the summation formula and exchanging the two finite sums yields:
%$$\sum_{j=1}^{m_k}\sum_{j'=1}^{m_k'}\sum_{(\xi,\zeta)\in\Xi}c_{(\xi,\zeta)} f_{k,j}(\xi) g_{k',j'}(\zeta)\overline{f_{k',j'}(\zeta')} \overline{g_{k,j}(\xi')}=0,$$
%for all $(k,k'
%)\in \mathcal{J}$ and any  $(\xi',\zeta') \in \M$.
%Applying the linear independence of the function $S_{k,j}(\xi')$, $F_{j,k}(\zeta')$, the latter implies 
%$$\sum_{(\xi,\zeta)\in\Xi}c_{\xi} f_{k,j}(\xi) g_{k',j'}(\zeta)=0,\quad \forall  (k,k')\in \mathcal{J}.$$
%The last statement contradicts the strict positive definiteness of $K$ as characterised in \Cref{THMSerKerSPD}.
%The second statement is again a direct consequence of the equivalence of (1) and (3) in \Cref{LEEquivspdGinv}.
%\end{proof}

%The properties for strictly positive definite isotropic kernels on products of manifold which we now utilize where derived in \citep{Guella2016b}, \citep{Guella2016a}, \citep{Guella2017b} and \citep{Barbosa2016}. 
For the case of products of spheres we additionally define the following subset of $\mathcal{J}$ and $\mathcal{F}$ respectively:
\begin{align*} \mathcal{J}^{0,0}&:=\mathcal{J} \cap \lbrace 2\N\times 2\N\rbrace\\
\mathcal{J}^{0,1}&:=\mathcal{J} \cap \lbrace 2\N\times (2\N+1)\rbrace \\
\mathcal{J}^{1,0}&:=\mathcal{J} \cap \lbrace (2\N+1)\times 2\N \rbrace \\
\mathcal{J}^{1,1}&:=\mathcal{J} \cap \lbrace (2\N+1)\times (2\N+1)\rbrace.
\end{align*} 
and for the circle: 
$\mathcal{J}_k:=\lbrace k' :\ (k,k') \in \mathcal{J}\rbrace$. 
\begin{corollary}
For a continuous positive definite kernel of the form \eqref{eqKernTwopointProdConv} on $\M\times\mathbb H$ we summarise the explicit conditions depending on the choices of $\M$ and $\mathbb H$:

\begin{enumerate}
\item Let  $\M=  \mathbb S^{d'-1}$ with $d'>2$, $\mathbb H\neq \mathbb S^{d-1}$ be a two-point homogeneous manifolds. For $K$ to be strictly positive definite it is necessary (sufficient) that $\mathcal{J}$ ($\mathcal{F}$) includes sequences $(k_r,k'_r)_{r=1}^{\infty}$ and $(j_r,j'_r)_{r=1}^{\infty}$ with $\lbrace k_r\rbrace \in 2\N$ and $\lbrace j_r\rbrace \in 2\N+1$ and $\underset{r\rightarrow \infty}{\lim}k_r=\underset{r\rightarrow \infty}{\lim}j_r=\underset{r\rightarrow \infty}{\lim}k'_r=\underset{r\rightarrow \infty}{\lim}j'_r=\infty$.
\item Let $\M, \mathbb H\neq \mathbb S^{d-1}$ be two-point homogeneous manifolds. For $K$ to be strictly positive definite it is necessary (sufficient) that $\mathcal{J}$ ($\mathcal{F}$) includes a sequences $(k_r,k'_r)_{r=1}^{\infty}$ with  $\underset{r\rightarrow \infty}{\lim}k_r=\underset{r\rightarrow \infty}{\lim}k'_r=\infty$.
\item Let $\M=\mathbb S^{d-1}$, $\mathbb H=\mathbb  S^{d'-1}$ with $d,d'>2$. For $K$ to be strictly positive definite it is necessary (sufficient) that $\mathcal{J}^{i,j}$ ($\mathcal{F}^{i,j}$)  includes  a sequences $({k_r}^{i,j},{k'_r}^{i,j})_{r=1}^{\infty}$ with  $\underset{r\rightarrow \infty}{\lim}k_r^{i,j}=\underset{r\rightarrow \infty}{\lim}{k'_r}^{i,j}=\infty$ for each pair $(i,j)\in \lbrace0,1\rbrace^2$.
\item Let $\M= \mathbb S^{d-1}$, $\mathbb H=\mathbb  S^{1}$. For $K$ to be strictly positive definite it is necessary (sufficient)  that  for each $\gamma \geq0$, the sets 
\begin{align*}
\lbrace k\in \Z: \mathcal{J}_{\vert k\vert} \cap\N_{\geq \gamma} \cap (2\Z+1) \neq \emptyset\rbrace,\\ 
\lbrace k\in \Z: \mathcal{J}_{\vert k\vert} \cap\N_{\geq \gamma} \cap (2\Z) \neq \emptyset\rbrace
\end{align*}
(or for the sufficient case
\begin{align*}
\lbrace k\in \Z:  \mathcal{F}_{\vert k\vert}  \cap \N_{\geq \gamma} \cap (2\Z+1) \neq \emptyset\rbrace,\\ 
\lbrace k\in \Z: \mathcal{F}_{\vert k\vert} \cap \N_{\geq \gamma} \cap (2\Z) \neq \emptyset\rbrace )
\end{align*}
intersect every full arithmetic progression.
\item Let $\M\neq\mathbb S^{d-1}$ be a two-point homogeneous manifold, $\mathbb H=\mathbb  S^{1}$. For $K$ to be strictly positive definite it is necessary (sufficient) that for each $\gamma>0$ the set  
$$\lbrace k\in \Z:\ \mathcal{J}_{\vert k\vert}   \cap \N_{\geq \gamma}\neq \emptyset\rbrace, $$
(or for the sufficient case $$\lbrace k\in \Z:\ \mathcal{F}_{\vert k\vert} \cap \N_{\geq \gamma}\neq \emptyset\rbrace, )$$
intersects every full arithmetic progression.
\item Let $\mathbb H=\M=\mathbb S^{1}$. For $K$ to be strictly positive definite it is necessary (sufficient) that the set 
$$\lbrace( k,l) : (\vert k\vert , \vert l\vert ) \in \mathcal{J} \rbrace$$  
(or for the sufficient case 
$$\lbrace( k,l) : (\vert k\vert , \vert l\vert ) \in \mathcal{F} \rbrace)$$
intersects all the translations of each subgroup
of $\Z^2$ having the form $(a, b)\Z + (0, d)\Z, a, d > 0.$
\end{enumerate}
\end{corollary}
\begin{proof}
From \Cref{LE:HomConNec} adapted to the product notation it follows that:
\begin{enumerate}
\item[a)] For  a continuous kernel of the form \eqref{eqKernTwopointProdConv} to be strictly positive definite it is necessary  that $\mathcal{J}$ induces strict positive definiteness for isotropic kernels on $\M \times \mathbb H$.
\item[b)] For  a continuous kernel given as \cref{eqKernTwopointProdConv} to be strictly positive definite it is sufficient that $\mathcal{F}$ induces strict positive definiteness for isotropc kernels on $\M\times \mathbb H$ and $a_{j,j',k,k'}\geq0$.
\end{enumerate}
The results now follow from the isotropic cases which were proven for $1.$  in Theorem 4.5 \citep{Barbosa2017}, $2.$ in Theorem 4.3  \citep{Barbosa2017}, $3.$  in Theorem 1.2  \citep{Guella2016a}, $4.$  in Theorem 3.11  \citep{Guella2016b}, $5.$  in Theorem 4.2 \citep{Guella2016b}, $6.$ in Theorem 1.2  \citep{Guella2017}.
\end{proof}

%\section{Conclusion}
%The results in section two unify the approaches used for the characterization mostly of $G$-invariant or isotropic kernels and establish common underlying concepts. The described results are helpful for the derivation of criteria of strict positive definiteness for more general kernel classes as described in Section 3. The additional specifications added in Section 3.1 and the examples show how now the abstract results allow to determine specific results.

\subsection*{Acknowledgement}
Jean Carlo Guella was funded by grant 2021/04226-0, S\~ao Paulo Research Foundation (FAPESP) 2021/04226-0. 
The work of J. J\"ager was supported by the DFG as part of the research project JA 3033/2-1 project number 461449252.

\def\ln{\log}

%\par\bigskip\bigskip\noindent
%\bibliographystyle{abbrvnat}
\bibliographystyle{amsplain}
\bibliography{C:/Users/Janin/JLUbox/Dokumente/LiteraturALL.bib}
%\bibliography{/home/janin/JLUbox/Dokumente/LiteraturALL}
%\bibliography{LiteraturALL}

\end{document}